\documentclass[a4paper,10pt]{article}
\usepackage[T1]{fontenc}
\usepackage{amssymb,amsmath,amsthm,mathtools,bbold, tikz,rotating,cite,todonotes, float}
\usepackage[all,cmtip]{xy}
\usepackage[enableskew]{youngtab}
\usepackage[left=2.5cm,top=2.5cm,bottom=2.5cm,right=2.5cm]{geometry}
\usepackage[british]{babel}
\usepackage{hyperref}

\newtheorem{lemma}{Lemma}[section]
\newtheorem{theorem}[lemma]{Theorem}
\newtheorem{corollary}[lemma]{Corollary}
\newtheorem{proposition}[lemma]{Proposition}

\theoremstyle{plain}
\newtheorem{defi}[lemma]{Definition}
\theoremstyle{definition}
\newtheorem{example}[lemma]{Example}

\DeclareMathOperator{\id}{Id}

\DeclareMathOperator{\GL}{GL}
\DeclareMathOperator{\reg}{Reg}
\DeclareMathOperator{\End}{End}
\DeclareMathOperator{\Tr}{Tr}
\DeclareMathOperator{\Irr}{Irr}
\DeclareMathOperator{\Pl}{Pl}
\DeclareMathOperator{\SPl}{SPl}
\DeclareMathOperator{\sch}{sch}
\DeclareMathOperator{\scl}{scl}

\DeclareMathOperator{\crs}{crs}
\DeclareMathOperator{\nst}{nst}
\DeclareMathOperator{\SInd}{SInd}
\DeclareMathOperator{\Res}{Res}
\DeclareMathOperator{\Sing}{Sing}
\DeclareMathOperator{\Leb}{Leb}

\newcommand{\R}{\mathbb{R}}
\newcommand{\C}{\mathbb{C}}

\newcommand{\N}{\mathbb{N}}
\newcommand{\K}{\mathcal{K}}

\title{A Plancherel measure associated to set partitions and its limit}
\author{Dario De Stavola \thanks{email: dario.destavola@math.uzh.ch\newline Institute of Mathematics, University of Zurich, Winterthurerstrasse 190, 8057, Zurich, Switzerland.\newline Keywords: supercharacter, set partition, Plancherel measure.}}
\begin{document}
\maketitle

\begin{abstract}
In recent years  increasing attention has been paid on the area of supercharacter theories, especially to those of the upper unitriangular group. A particular supercharacter theory, in which supercharacters are indexed by set partitions, has several interesting properties, which make it object of further study. We define a natural generalization of the Plancherel measure, called superplancherel measure, and prove a limit shape result for a random set partition according to this distribution. We also give a description of the asymptotical behavior of two set partition statistics related to the supercharacters. The study of these statistics when the set partitions are uniformly distributed has been done by Chern, Diaconis, Kane and Rhoades. 
\end{abstract}

\section{Introduction}
Let $p$ be a prime number, $q$ a power of $p$, and $\mathbb{K}$ the finite field of order $q$ and characteristic $p$. Consider $U_n=U_n(\mathbb{K})$ to be the group of upper unitriangular matrices with entries in $\mathbb{K}$, it is known that the description of conjugacy classes and irreducible characters of $U_n$ is a wild problem, in the sense described, for example, by Drodz in \cite{drozd1980tame}. To bypass the issue, Andr{\'e} \cite{andre1995basic} and Yan \cite{yan2010representations} set the foundations of what is now known as ``supercharacter theory'' (in the original works it was called ``basic character theory''). The idea is to meld together some irreducible characters and conjugacy classes (called respectively supercharacters and superclasses), in order to have characters which are easy enough to be tractable but still carry information of the group. In particular, one obtains a smaller character table, which is required to be a square matrix. As an application, in \cite{arias2004super}, Arias-Castro, Diaconis and Stanley described random walks on $U_n$ utilizing only the supercharacter table (usually the complete character table is required). In \cite{diaconis2008supercharacters}, Diaconis and Isaacs formalized the axioms of supercharacter theory, generalizing the construction from $U_n$ to algebra groups.\smallskip

Among the various supercharacter theories for $U_n$ a particular nice one, hinted in \cite{aguiar2012supercharacters} and described by Bergeron and Thiem in \cite{bergeron2013supercharacter}, has the property that the supercharacters take integer values on superclasses. This is particularly interesting because of a result of Keller \cite{keller2014generalized}, who proves that for each group $G$ there exists a unique finest supercharacter theory with integer values. Although it is not yet known if Bergeron and Thiem's theory is the finest integral one, it has remarkable properties which make it worth of a deeper analysis. In this theory the supercharacters of $U_n$ are indexed by set partitions of $\{1,\ldots,n\}$ and they form a basis for the Hopf algebra of superclass functions. This Hopf algebra is isomorphic to the algebra of symmetric functions in noncommuting variables. Moreover, the supercharacter table decomposes as the product of a lower triangular matrix and an upper triangular matrix.
\smallskip

In the theory introduced by Bergeron and Thiem, the characters depend on the following three statistics defined for a set partition $\pi$ of $[n]$:
\begin{itemize}
\item $d(\pi)$, the number of arcs of $\pi$;
 \item$\dim(\pi)$, that is, the sum $\sum \max(B)-\min(B)$, where the sum runs through the blocks $B$ of $\pi$;
 \item $\crs(\pi)$, the number of crossings of $\pi$.
\end{itemize}
More precisely, we have that if $\chi^{\pi}$ is the supercharacter associated to the set partition $\pi$ then the dimension is $\chi^{\pi}(1)=q^{\dim(\pi)-d(\pi)}$ and $\langle\chi^{\pi},\chi^{\pi}\rangle=q^{\crs(\pi)}$. 

In the setting of probabilistic group theory one is interested in the study of statistics of the ``typical'' irreducible representation of the group. A natural probability distribution is the uniform distribution; in \cite{chern2014closed} and \cite{chern2015central} Chern, Diaconis, Kane and Rhoades study the statistics $\dim$ and $\crs$ for a uniform random set partition, proving formulas for the moments of $\dim(\pi)$ and $\crs(\pi)$ and, successively, a central limit theorem for these two statistics. These results imply that, for a uniform random set partition $\pi$ of $n$,
\[\dim(\pi)-d(\pi)= \frac{\alpha_n-2}{\alpha_n}n^2+O_P\left(\frac{n}{\alpha_n}\right),\qquad \crs(\pi)= \frac{2\alpha_n-5}{4\alpha_n^2}n^2+O_P\left(\frac{n}{\alpha_n}\right),\]
where $\alpha_n$ is the positive real solution of $u e^u=n+1$, so that $\alpha_n=\log n-\log\log n+o(1).$
\smallskip

In representation theory another natural distribution is the Plancherel measure, which is a discrete probability measure associated to the irreducible characters of a finite group. The Plancherel measure has received vast coverage in the literature, especially in the case of the symmetric group $S_n$. Since the irreducible characters of $S_n$ are indexed by the partitions of $n$, the problem of investigating longest increasing subsequences of a uniform random permutation is equivalent to studying the first rows of a Plancherel distributed integer partition (see \cite{romik2015surprising}). This prompted the study of asymptotics of the Plancherel measure, and in 1977 a limit shape result for a random partition was proved independently by Kerov and Vershik \cite{kerov1977asymptotics} and Logan and Shepp \cite{LoganShepp1977}. The result was later improved to a central limit theorem  by Kerov \cite{ivanov2002kerov}. Moreover, it was proved by Borodin, Okounkov and Olshanski \cite{BorodinOkounkovOlshanski2000} that the rescaled limiting distribution of the first $k$ rows of an integer partition coincides with the one of the largest $k$ eigenvalues of a GUE random matrix.
\smallskip

From the study of the Plancherel measure of $S_n$ has followed a theory regarding the Plancherel growth process. Indeed, there exist natural transition measures between partitions of $n$ and partitions of $n+1$, which generate a Markov process whose marginals are the Plancherel distributions. The transition measures have a nice combinatorial description, see \cite{kerov1993transition}.
\medskip

In this paper we generalize the notion of Plancherel measure  to adapt it to supercharacter theories. We call the measure associated to a supercharacter theory \emph{superplancherel measure}. We show that for a tower of groups $\{1\}=G_0\subseteq G_1\subseteq\ldots$, each group endowed with a consistent supercharacter theory, there exists a nontrivial transition measure which yields a Markov process; the marginals of this process are the superplancherel measures. In order to do so, we generalize a construction of superinduction for algebra groups to general finite groups. Such a construction was introduced by Diaconis and Isaacs in \cite{diaconis2008supercharacters} and developed by Marberg and Thiem in \cite{marberg2009superinduction}.
\smallskip

We then consider the superplancherel measure associated to the supercharacter theory of $U_n$ described by Bergeron and Thiem. In this setting, the superplancherel measure has an explicit formula depending on the statistics $\dim(\pi)$ and $\crs(\pi)$; we give a direct combinatorial construction of such a measure.
\smallskip

The main result of the paper is a limit shape for a random superplancherel distributed set partition. In order to formulate this result we immerse set partitions into the space of subprobabilities (\emph{i.e.}, measures with total weight less than or equal to $1$) of the unit square $[0,1]^2$ with some other properties. This embedding is similar to that of permutons for random permutations, see for example \cite{glebov2015finitely}. Given a set partition $\pi$ we refer to the corresponding subprobability as $\mu_{\pi}$. We describe a measure $\Omega$ such that
\begin{theorem}\label{main result 1}
  For each $n\geq 1$ let $\pi_n$ be a random set partition of $n$ distributed with the superplancherel measure $\SPl_n$, then
 \[\mu_{\pi_n}\to\Omega\mbox{ almost surely} \]
 where the convergence is the weak* convergence for measures.
\end{theorem}
The measure $\Omega$ is the uniform measure on the set $\{(x,1-x)\mbox{ s.t. }x\in[0,1/2]\}$ of total weight $1/2$. Informally, we can say that a set partition chosen at random with the superplancherel measure is asymptotically close to the the following shape:
 \[
 \begin{tikzpicture}[scale=0.5]
\draw [fill] (0,0) circle [radius=0.05];
\node [below] at (0,0) {1};
\draw [fill] (1,0) circle [radius=0.05];
\node [below] at (1,0) {2};
\draw [fill] (2,0) circle [radius=0.05];
\node [below] at (2,0) {3};
\draw [fill] (3.5,0) circle [radius=0.03];
\draw [fill] (3.75,0) circle [radius=0.03];
\draw [fill] (4,0) circle [radius=0.03];
\draw [fill] (5.5,0) circle [radius=0.05];
\draw [fill] (6.5,0) circle [radius=0.05];
\draw [fill] (7.5,0) circle [radius=0.05];
\node [below] at (7.5,0) {n};
\draw (0,0) to[out=70, in=110] (7.5,0);
\draw (1,0) to[out=70, in=110] (6.5,0);
\draw (2,0) to[out=70, in=110] (5.5,0);
\end{tikzpicture}\]
In the process, we obtain asymptotic results for the statistics $\dim(\pi)$ and $\crs(\pi)$ when $\pi$ is chosen at random with the superplancherel measure:
\begin{corollary}\label{main result 2}
  For each $n\geq 1$ let $\pi_n$ be a random set partition of $n$ distributed with the superplancherel measure $\SPl_n$, then
 \[\frac{\dim(\pi)}{n^2}\to\frac{1}{4}\mbox{ a.s.},\qquad \crs(\pi)\in O_P(n^2).\] 
\end{corollary}
These results will be prove
\smallskip

As mentioned, the main idea is to consider set partitions as particular measures of the unit square. With this transformation the statistics $\dim(\pi)$ and $\crs(\pi)$ can be seen as integrals of the measure $\mu_{\pi}$. We use an entropy argument to delimitate a set of set partitions of maximal probability. Finally, we relate the results on the entropy into the weak* topology of measures of $[0,1]^2.$
\smallskip

The combinatorial interpretation of the superplancherel measure for $U_n$ allows us to have computer generated superplancherel random set partitions $\pi\vdash[n]$ for fairly large $n$. In Figure \ref{fig:boat1} we present one of such $\mu_{\pi}$ for $\pi\vdash[200]$; we observe that it is indeed closed to $\Omega$.
\begin{figure}[H]
  \begin{center}
  \[\includegraphics[width=4.3cm,height=4.3cm]{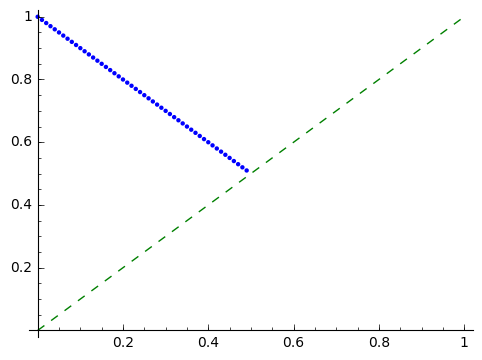}\qquad\qquad
\includegraphics[width=4.3cm,height=4.3cm]{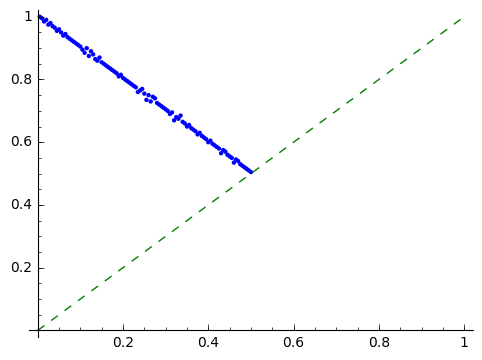}\]
  \caption{Description of a random superplancherel distributed set partition: the left image is the measure $\mu_{\Omega}$ associated to the set partition $\Omega$; on the right there is a computer generated random measure $\mu_{\pi}$ for $\pi\vdash[200]$. The algorithm we use for the program that generates a big random set partition is based on the combinatorial interpretation in Section \ref{section combinatorial interpretation}.}\end{center}
  \label{fig:boat1}
\end{figure}

\subsection{Outline of the paper}
In section \ref{section preliminries} we recall some basic notions of representation theory and supercharacter theory; we define the superplancherel measure and a transition measure; in section \ref{section supercharacter theory} we define the most important statistics of set partitions in the topic of the supercharacter theory of $U_n$, we find an explicit formula for the superplancherel measure, and we give a combinatorial interpretation. In section \ref{section set partitions as measures} we see set partitions as measures in $[0,1]^2$ and we study the statistics $\dim(\pi)$ and $\crs(\pi)$ in this setting. Finally, in section \ref{section final results} we prove the limit shape result for random set partitions and the result on the asymptotic behavior of $\dim(\pi)$ and $\crs(\pi)$ (respectively Theorem \ref{main result 1} and Corollary \ref{main result 2}.

\section{Preliminaries}\label{section preliminries}
\subsection{Reminders on character theory}
\begin{defi}
 Let $G$ be a finite group and $V$ a finite dimensional $\C$-vector space. An homomorphism $\pi\colon G\to \GL(V)$ is called a \emph{$\C$-linear representation}.
\end{defi}
The function $\pi$ can be extended by linearity to a homomorphism of $\C$-algebras, $\pi\colon \C G\to\End_{\C}(V)$. In this way we obtain an action of the group algebra $\C G$ on $V$, thus defining a left $\C G$-module $(V,\pi)$. We will refer to it simply as $V$ if the action is clear from the context.
\begin{defi}
 Given a $\C G$-module $(V,\pi)$ the character afforded by $V$ is the map $\xi_V\colon G\to\C$ defined by $\xi_V(g)=\Tr(g,V)$, that is, the trace of the representation $\pi(g)$ seen in matrix form. Indeed it is well known that the character does not depend on the basis chosen for $V$.
\end{defi}
The value $\xi_V(1)=\Tr(1,V)$ is equal to $\dim_{\C}(V)$ and it is called the \emph{degree} of the representation.
\smallskip

 Let $(V,\pi)$ be a $\C G$-module, then $(W,\pi)$ is a \emph{submodule} if it is a module and $W\subseteq V$. The module $(V,\pi)$ is said to be \emph{irreducible} if the only submodules of $V$ are $\{0\}$ and $V$. By Manschke's Theorem every $\C G$-module decomposes as a direct sum of irreducible $\C G$-modules. We can choose a set of irreducible modules $V_1,\ldots,V_t$ affording the characters $\xi_1,\ldots,\xi_t$ respectively. We define the set
 \[\Irr(G):=\{\xi_1,\ldots,\xi_t\}\]
 which is independent from the choice of the irreducible modules $V_1,\ldots,V_t$. In particular we can see $\C G$ as a module acting on itself. The character $\rho_G$ afforded by $\C G$ is called the \emph{regular character} and decomposes thus:
\[\rho_G=\sum_{\xi\in\Irr(G)}\xi(1)\xi;\]
it is easy to see that $\rho_G(g)=|G|\delta_{g,1}$, where $\delta_{x,y}$ is the Kronecher delta, equal to $1$ if $x=y$ and $0$ otherwise. As a consequence we obtain the formula
\[|G|=\sum_{\xi\in\Irr(G)}\xi(1)^2.\]
This leads to a discrete probability measure on the set of irreducible characters, called the \emph{Plancherel measure}: $\Pl_G(\xi)=\xi(1)^2/|G|$.

We recall the \emph{Frobenius scalar product} between two functions $\chi,\zeta\colon G\to \C$:
\[\langle\chi,\zeta\rangle:=\frac{1}{|G|}\sum_{g\in G}\chi(g)\overline{\zeta(g)}.\]
Characters form an orthonormal basis, with respect to the Frobenius scalar product, of the algebra of class functions, \emph{i.e.} complex valued functions which take constant values on conjugacy classes. This property is known as the character orthogonality relations of the first kind: if $\xi_1,\xi_2\in\Irr(G)$ then $\langle \xi_2,\xi_2\rangle=\delta_{\xi_1,\xi_2}$.
If $\chi,\xi$ are characters of $G$ and $\xi$ is irreducible we say that $\xi$ is a \emph{constituent} of $\chi$ if $\langle \chi,\xi\rangle\neq 0$. Moreover, we call $I(\chi):=\{\xi\in\Irr(G)\mbox{ s.t. }\langle\chi,\xi\rangle\neq 0\}$. It is immediate to see that for each character $\chi$ of $G$ we have
\[\chi=\sum_{\xi\in\Irr(G)}\langle\chi,\xi\rangle\xi=\sum_{\xi\in I(\chi)}\langle\chi,\xi\rangle\xi\]
\subsection{Supercharacter theory}
We recall the definition of supercharacter theory. Notice that this definition coincides with the one introduced in \cite{diaconis2008supercharacters} due to \cite[Lemma 2.1]{diaconis2008supercharacters}.
\begin{defi}\label{definition supercharacter theory}
 A \emph{supercharacter theory} of a finite group $G$ is a pair $(\scl(G),\sch(G))$ where $\scl(G)$ is a set partition of $G$ and $\sch(G)$ an orthogonal set of nonzero characters of $G$ (not necessarily irreducible) such that:
 \begin{enumerate}
  \item $|\scl(G)|=|\sch(G)|$;
  \item every character $\chi\in\sch(G)$ takes a constant value on each member $\mathcal{K}\in\scl(G)$;
  \item each irreducible character of $G$ is a constituent of one, and only one, of the characters $\chi\in\sch(G)$.
 \end{enumerate}
\end{defi}
The elements $\mathcal{K}\in\scl(G)$ are called \emph{superclasses}, while the characters $\chi\in\sch(G)$ are \emph{supercharacters}. It is easy to see that every element $\mathcal{K}\in \scl(G)$ is always a union of conjugacy classes. Since a supercharacter $\chi\in\sch(G)$ is always constant on superclasses we will sometimes write $\chi(\mathcal{K})$ instead of $\chi(g)$, where $\mathcal{K}\in\scl(G)$ is a superclass and $g\in\mathcal{K}$. Observe that irreducible character theory is a supercharacter theory.
\begin{example}
For every finite group $G$ there are two trivial supercharacter theories.
\begin{itemize}
 \item The irreducible character theory, where $\scl(G)$ is the set of conjugacy classes of $G$ and $\sch(G)$ is the set of irreducible characters.
 \item The supercharacter theory where $\scl(G)=\{\{1_G\},G\setminus\{1_G\}\}$ and $\sch(G)=\{\id_G,\rho_G-\id_G\}$.
\end{itemize}
\end{example}
\subsection{Superplancherel measure}
\begin{defi}\label{definition of superplancherel}
Fix a supercharacter theory $T=(\scl(G),\sch(G))$ of $G$, we define the \emph{superplancherel measure} $\SPl_G$ of $T$ as follow: given $\chi\in\sch(G)$, then $\SPl_G^T(\chi):=\frac{1}{|G|}\frac{\chi(1)^2}{\langle\chi,\chi\rangle}.$
\end{defi}
Notice that if $T$ is the irreducible character theory, then the superplancherel measure is equal to the usual Plancherel measure. We stress out that the definition of superplancherel measure depends on the supercharacter theory but we will omit it if it is clear from the context.
\smallskip

Let us show that $\SPl_G$ is indeed a probability measure; we prove first supercharacter orthogonality relations of first and second kind. Fix a supercharacter theory $T=(\scl(G),\sch(G))$ for $G$, then by \cite[Lemma 2.1]{diaconis2008supercharacters} we know that for every supercharacter $\chi\in \sch(G)$ there exists $c(\chi)\in \C$ such that
\begin{equation}\label{equation supercharacter formula into irr characters}
 c(\chi)\chi=\sum_{\xi\in I(\chi)}\xi(1)\xi.
\end{equation}
\begin{proposition}
 Set $\chi_1,\chi_2\in\sch(G)$, then
 \[\langle \chi_1,\chi_2\rangle=\frac{\chi_1(1)}{c(\chi_1)}\delta_{\chi_1,\chi_2}. \]
\end{proposition}
\begin{proof}
 Consider
 \[c(\chi_1)\chi_1=\sum_{\xi\in I(\chi_1)}\xi(1)\xi,\qquad c(\chi_2)\chi_2=\sum_{\xi\in I(\chi_2)}\xi(1)\xi.\]
 Then
 \[\langle \chi_1,\chi_2\rangle=\frac{1}{c(\chi_1)c(\chi_2)}\sum_{\substack{\xi_1\in I(\chi_1)\\\xi_2\in I(\chi_2)}}\xi_1(1)\xi_2(1)\langle\xi_1,\xi_2\rangle.\]
 By the first orthogonality relations we have that $\langle\xi_1,\xi_2\rangle=\delta_{\xi_1,\xi_2}$; but if $\chi_1\neq\chi_2$ then $I(\chi_1)\cap I(\chi_2)=\emptyset$ by the third property of Definition \ref{definition supercharacter theory}. Hence 
 \[\langle \chi_1,\chi_2\rangle=\frac{1}{c(\chi_1)c(\chi_2)}\sum_{\substack{\xi_1\in I(\chi_1)\\\xi_2\in I(\chi_2)}}\xi_1(1)\xi_2(1)\delta_{\xi_1,\xi_2}=0\]
 if $\chi_1\neq\chi_2$. On the other hand, if $\chi_1=\chi_2$ then
 \[\langle \chi_1,\chi_2\rangle=\frac{1}{c(\chi_1)c(\chi_2)}\sum_{\substack{\xi_1\in I(\chi_1)\\\xi_2\in I(\chi_2)}}\xi_1(1)\xi_2(1)\delta_{\xi_1,\xi_2}=\frac{1}{c(\chi_1)^2}\sum_{\xi_1\in I(\chi_1)}\xi_1(1)^2=\frac{\chi_1(1)}{c(\chi_1)}.\]
Therefore we can conclude that $\langle \chi_1,\chi_2\rangle=\frac{\chi_1(1)}{c(\chi_1)}\delta_{\chi_1,\chi_2}$.
\end{proof}
In the irreducible character theory, a direct consequence of the orthogonality relations of the first kind is the orthogonality relations of the second kind: if $g,h\in G$ then
\begin{equation}\label{orthogonality relations of the second kind}
 \sum_{\xi\in\Irr(G)}\xi(g)\overline{\xi(h)}=\frac{|G|}{|\mathcal{C}_g|}\delta_{\mathcal{C}_g,\mathcal{C}_h},
\end{equation}
where $\mathcal{C}_g,\mathcal{C}_h$ are the conjugacy classes of respectively $g$ and $h$. We adapt the proof of this result to the supercharacter theory, see for example \cite[Theorem 1.10.3]{sagan2013symmetric}
\begin{proposition}
 Let $\K_1,\K_2\in\scl(G)$, then
 \[\sum_{\chi\in\sch(G)}\frac{c(\chi)}{\chi(1)}\chi(\K_1)\overline{\chi(\K_2)}=\frac{|G|}{|\K_1|}\delta_{\K_1,\K_2}.\]
\end{proposition}
\begin{proof}
The modified supercharacter table
 \[U=\left[ \sqrt{\frac{|\K|}{|G|}}\sqrt{\frac{c(\chi)}{\chi(1)}}\chi(\K)\right]_{\chi\in\sch(G),\K\in\scl(G)}\]
 is unitary, that is, it has orthonormal rows, due to the previous proposition. This implies that it has also orthonormal columns, \emph{i.e.},
 \[\sum_{\chi\in\sch(G)}\frac{c(\chi)}{\chi(1)}\frac{\sqrt{|\K_1|}\sqrt{|\K_2|}}{|G|}\chi(\K_1)\overline{\chi(\K_2)}=\delta_{\K_1,\K_2}.\qedhere\]
\end{proof}
\begin{proposition}
 For each group $G$ and supercharacter theory $T$ of $G$ the superplancherel measure $\SPl_G$ is a probability measure.
\end{proposition}
\begin{proof}
Since 
\[\langle\chi,\chi\rangle=\frac{1}{c(\chi)^2}\sum_{\xi\in I(\chi)}\xi(1)^2=\frac{\chi(1)}{c(\chi)},\]
then $\SPl_G^T(\chi):=\frac{1}{|G|}\frac{\chi(1)^2}{\langle\chi,\chi\rangle}= \frac{c(\chi)}{|G|}\chi(1).$
Another consequence of \cite[Lemma 2.1]{diaconis2008supercharacters} is that $\K=\{1\}$ is always a superclass. In particular the previous proposition applied to $\K_1=\{1\}=\K_2$ gives:
\[\sum_{\chi\in \sch(G)}\frac{c(\chi)}{|G|}\chi(1)=1\]
hence the superplancherel measure is indeed a probability measure.
\end{proof}                                                                                                      
\subsection{Superinduction and transition measure}\label{section superinduction}
In this section we extend the notion of Superinduction, defined by Diaconis and Isaacs in \cite{diaconis2008supercharacters} for algebra groups, to general finite groups, and we use it to define a transition measure. Let $G$ be a finite group, $H\leq G$ a subgroup and $(\scl(G),\sch(G))$ a supercharacter theory for $G$. Let $\phi\colon H\to\C$ be any function, we set $\phi^0\colon G\to\C$ to be $\phi^0(g)=\phi(g)$ if $g\in H$ and $\phi^0(g)=0$ otherwise. We define
\[\SInd_H^G(\phi)(g):=\frac{|G|}{|H|\cdot|[g]|}\sum_{k\in[g]}\phi^0(k),\]
where $[g]\in\scl(G)$ is the superclass containing $g$. By construction, $\SInd_H^G(\phi)$ is a superclass function. Since $\sch(G)$ is an orthogonal basis for the algebra of superclass functions (see \cite[Theorem 2.2]{diaconis2008supercharacters}), we can expand $\SInd_H^G(\phi)$ in this basis:
\[\SInd_H^G(\phi)=\sum_{\chi\in\sch(G)}\frac{\langle \SInd_H^G(\phi),\chi\rangle}{\langle\chi,\chi\rangle}\chi.\]

A supercharacter version of the Frobenius reciprocity holds: if $\psi$ is a superclass function then
\begin{align*}
 \langle \SInd_H^G(\phi),\psi\rangle&=\frac{|G|}{|H|\cdot|[g]|}\frac{1}{|G|}\sum_{g\in G}\sum_{k\in[g]}\overline{\phi^0(k)}\psi(g)\\
 &=\frac{1}{|H|}\sum_{\mathcal{K}\in\scl(G)}\sum_{g,k\in\mathcal{K}}\frac{\overline{\phi^0(k)}\psi(k)}{|\mathcal{K}|}\\
 &=\frac{1}{|H|}\sum_{\mathcal{K}\in\scl(G)}\sum_{k\in\mathcal{K}}\overline{\phi^0(k)}\psi(k)\\
 &=\frac{1}{|H|}\sum_{k\in G}\overline{\phi^0(k)}\psi(k)\\
 &=\frac{1}{|H|}\sum_{k\in H}\overline{\phi(k)}\psi(k)=\langle\phi,\Res_H^G(\psi)\rangle.
\end{align*}
Here $\Res_H^G(\psi)$ is the restriction of $\psi$ to $H$.
\smallskip

Consider now also $H$ endowed with a supercharacter theory $(\scl(H),\sch(H))$. Suppose also that this supercharacter theory is \emph{consistent} with the one of $G$, that is, for each $\mathcal{H}\in\scl(H)$ there exists $\mathcal{K}\in\scl(G)$ such that $\mathcal{H}\subseteq\mathcal{K}$. This is equivalent to the requirement that $\Res_H^G(\chi)$ is a superclass function on $H$ for each $\chi\in\sch(G)$ by \cite[Theorem 2.2]{diaconis2008supercharacters}.
\begin{defi}
 Let $\chi\in\sch(G)$, $\gamma\in\sch(H)$. The \emph{transition measure} $\rho_H^G(\gamma,\chi)$ is defined as
 \[\rho_H^G(\gamma,\chi):=\frac{|H|}{|G|}\frac{\chi(1)}{\gamma(1)}\frac{\langle \SInd_H^G(\gamma),\chi\rangle}{\langle\chi,\chi\rangle}.\] 
\end{defi}
\begin{proposition}
 The following hold:
 \begin{enumerate}
  \item  For each $\chi\in\sch(G)$ we have $\sum_{\gamma\in\sch(H)}\rho_H^G(\gamma,\chi)\SPl_H(\gamma)=\SPl_G(\chi)$.
  \item For each $\gamma\in\sch(H)$ we have $\sum_{\chi\in\sch(G)}\rho_H^G(\gamma,\chi)=1$.
 \end{enumerate} 
In particular, let $\{1\}=G_0\hookrightarrow\ldots\hookrightarrow G_n\hookrightarrow G_{n+1}\hookrightarrow\ldots$ be a tower of groups, and suppose that for each $n$ we associate a supercharacter theory $T_n$ to $G_n$ which is consistent with $T_{n+1}$. Let $\chi_1$ be the unique supercharacter for $G_0=\{1\}$; consider the Markov process with initial state $\chi_1$ and transition measures $\rho_{U_n}^{U_{n+1}}$. Then this process has marginals distributed as $\SPl_{G_n}$.
 \end{proposition}
\begin{proof}
 \begin{enumerate}
  \item Set $\chi\in\sch(G)$, then
  \begin{align*}
   \sum_{\gamma\in\sch(H)}\rho_H^G(\gamma,\chi)\SPl_H(\gamma)&=\sum_{\gamma\in\sch(H)}\frac{|H|}{|G|}\frac{\chi(1)}{\gamma(1)}\frac{\langle \SInd_H^G(\gamma),\chi\rangle}{\langle\chi,\chi\rangle}\frac{1}{|H|}\frac{\gamma(1)^2}{\langle\gamma,\gamma\rangle}\\
   &=\frac{\chi(1)}{|G|\langle\chi,\chi\rangle}\sum_{\gamma\in\sch(H)}\frac{\langle \gamma,\Res_H^G(\chi)\rangle}{\langle\gamma,\gamma\rangle}\gamma(1)\\
   &=\frac{\chi(1)}{|G|\langle\chi,\chi\rangle}\Res_H^G(\chi)(1)\\
   &=\frac{\chi(1)^2}{|G|\langle\chi,\chi\rangle}=\SPl_G(\chi).
  \end{align*}
  Notice that in the third equality we used the fact that $\Res_H^G(\chi)$ is a superclass function, and thus can be written as $\Res_H^G(\chi)=\sum_{\gamma}\frac{\langle \gamma,\Res_H^G(\chi)\rangle}{\langle\gamma,\gamma\rangle}\gamma$.
\item Set $\gamma\in\sch(H)$, then 
\[\sum_{\chi\in\sch(G)}\rho_H^G(\gamma,\chi)=\sum_{\chi\in\sch(G)}\frac{|H|}{|G|}\frac{1}{\gamma(1)}\frac{\langle \SInd_H^G(\gamma),\chi\rangle}{\langle\chi,\chi\rangle}\chi(1)=\frac{|H|}{|G|}\frac{1}{\gamma(1)}\SInd(\gamma)(1)=1.\qedhere\]  
 \end{enumerate} 
\end{proof}
\section{Supercharacter theory for unitriangular matrices}\label{section supercharacter theory}
\subsection{Set partitions}
We recall some basic definitions regarding set partitions. Let $n\in\N$ and set $[n]$ to be the set $\{1,\ldots,n\}$. A \emph{set partition} $\pi$ of $[n]$ is a family of non empty sets, called \emph{blocks}, which are disjoint and whose union is $[n]$. If $\pi$ is a set partition of $[n]$ we write $\pi\vdash[n]$. Conventionally the blocks of $\pi$ are ordered by increasing value of the smallest element of the block, and inside every block the elements are ordered with the usual order of natural numbers. If two numbers $i$ and $j$ are in the same block of the set partition $\pi\vdash[n]$ and there is no $k$ in that block such that $i<k<j$, then the pair $(i,j)$ is said to be an \emph{arc} of $\pi$. The set partition $\pi$ is uniquely determined by the set $D(\pi)$ of arcs. The \emph{standard representation} of $\pi\vdash[n]$ is the graph with vertex set $[n]$ and edge set $D(\pi)$, as in Figure \ref{picture standard representation}.
\begin{figure}[H]
 \[ \pi= \begin{array}{c}
 \begin{tikzpicture}[scale=0.5]
\draw [fill] (0,0) circle [radius=0.05];
\node [below] at (0,0) {1};
\draw [fill] (1,0) circle [radius=0.05];
\node [below] at (1,0) {2};
\draw [fill] (2,0) circle [radius=0.05];
\node [below] at (2,0) {3};
\draw [fill] (3,0) circle [radius=0.05];
\node [below] at (3,0) {4};
\draw [fill] (4,0) circle [radius=0.05];
\node [below] at (4,0) {5};
\draw [fill] (5,0) circle [radius=0.05];
\node [below] at (5,0) {6};
\draw [fill] (6,0) circle [radius=0.05];
\node [below] at (6,0) {7};
\draw [fill] (7,0) circle [radius=0.05];
\node [below] at (7,0) {8};
\draw [fill] (8,0) circle [radius=0.05];
\node [below] at (8,0) {9};
\draw (0,0) to[out=70, in=110] (4,0);
\draw (4,0) to[out=70, in=110] (6,0);
\draw (2,0) to[out=70, in=110] (3,0);
\draw (3,0) to[out=70, in=110] (8,0);
\draw (5,0) to[out=70, in=110] (7,0);
\end{tikzpicture}\end{array} \]
\caption{Example of the partition $\pi=\left\{\{1,5,7\},\{2\},\{3,4,9\},\{6,8\}\right\}\vdash[9]$ in standard representation.}\label{picture standard representation}
\end{figure}
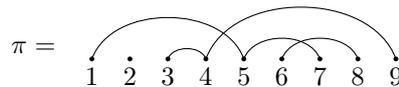
Fix $\pi\vdash[n]$, then define
\begin{itemize}
 \item the \emph{dimension} $\dim(\pi)$ as 
 \[\dim(\pi):=\sum_{(i,j)\in D(\pi)}j-i.\]
 For example, in the set partition of Figure \ref{picture standard representation}, the dimension is $\dim(\pi)=14$.
 \item The number of crossings $\crs(\pi)$ of $\pi$, where a \emph{crossing} is an unordered pair of arcs $\{(i,j),(k,l)\}$ in $D(\pi)$ such that $i<k<j<l$. Diagrammatically a crossing corresponds to the picture 
 \begin{figure}[H]
 \[ \begin{array}{c}
 \begin{tikzpicture}[scale=0.5]
\draw [fill] (0,0) circle [radius=0.05];
\node [below] at (0,0) {$i$};
\draw [fill] (1,0) circle [radius=0.05];
\node [below] at (1,0) {$k$};
\draw [fill] (2,0) circle [radius=0.05];
\node [below] at (2,0) {$j$};
\draw [fill] (3,0) circle [radius=0.05];
\node [below] at (3,0) {$l$};
\draw (0,0) to[out=70, in=110] (2,0);
\draw (1,0) to[out=70, in=110] (3,0);
\end{tikzpicture}\end{array} \]
\end{figure}
 In the example of Figure \ref{picture standard representation}, the number of crossings of $\pi$ is $\crs(\pi)=2$. 
 \item The number of nestings $\nst(\pi)$, where a \emph{nesting} is an unordered pair of arcs $\{(i,j),(k,l)\}\subseteq D(\pi)$ such that $i<k<l<j$. Diagrammatically a crossing corresponds to the picture 
 \begin{figure}[H]
 \[ \begin{array}{c}
 \begin{tikzpicture}[scale=0.5]
\draw [fill] (0,0) circle [radius=0.05];
\node [below] at (0,0) {i};
\draw [fill] (1,0) circle [radius=0.05];
\node [below] at (1,0) {$k$};
\draw [fill] (2,0) circle [radius=0.05];
\node [below] at (2,0) {$l$};
\draw [fill] (3,0) circle [radius=0.05];
\node [below] at (3,0) {$j$};
\draw (0,0) to[out=70, in=110] (3,0);
\draw (1,0) to[out=70, in=110] (2,0);
\end{tikzpicture}\end{array} \]
\end{figure}
 In the example of Figure \ref{picture standard representation}, the number of nestings of $\pi$ is $\nst(\pi)=3$. 
\end{itemize}
Fix $\pi\vdash[n]$ and $i,j$ with $i<j\leq n$. The pair $(i,j)$ is said to be $\pi$-regular if there exists no $k<i$ such that $(k,j)\in D(\pi)$ and there exists no $l>j$ such that $(i,l)\in D(\pi)$. The set of $\pi$-regular pairs is denoted $\reg(\pi)$. For example, given $\pi=\{\{1,4\},\{2,3,5\}\}=\begin{array}{c}\begin{tikzpicture}[scale=0.4]
\draw [fill] (0,0) circle [radius=0.05];\draw [fill] (1,0) circle [radius=0.05];\draw [fill] (2,0) circle [radius=0.05];\draw [fill] (3,0) circle [radius=0.05];\draw [fill] (4,0) circle [radius=0.05];
\draw (0,0) to[out=70, in=110] (3,0);
\draw (1,0) to[out=70, in=110] (2,0);
\draw (2,0) to[out=70, in=110] (4,0);
\end{tikzpicture}\end{array}$ then the set $\reg(\pi)$ is $\{(1,4),(1,5),(2,3),(2,5),(3,5)\}$; if an arc is not regular then it is called \emph{singular} and the set of $\pi$-singular arcs is denoted $\Sing(\pi)$. In the previous example thus $\Sing(\pi)=\{(1,2),(1,3),(2,4),(3,4),(4,5)\}$.
\smallskip

If $\pi\vdash[n]$ and $k<l\leq n$ then $\nst_{\pi}(k,l)=\sharp\{(i,j)\in D(\pi)\mbox{ s.t. }i<k<l<j\}$. If $\sigma\vdash[m]$ with $m\leq n$ then 
\[\nst_{\pi}(\sigma):=\sum_{(k,l)\in D(\sigma)}\nst_{\pi}(k,l).\]

\subsection{A supercharacter theory for \texorpdfstring{$U_n$}{Un}}
Let $\mathbb{K}$ be the finite field of order $q$ and characteristic $p$. The group $U_n=U_n(\mathbb{K})$ is the group of upper unitriangular matrices of size $n\times n$ and entries belonging to $\mathbb{K}$, that is,
\[U_n=U_n(\mathbb{K})=\left\{\left[\begin{array}{cccc}
1&a_{1,2}&\cdots&a_{1,n}\\&1&a_{2,3}&\vdots\\&&\ddots&a_{n-1,n}\\&&&1\\
\end{array}\right]\in M_{n\times n}(\mathbb{K})\right\}.\]
In \cite{bergeron2013supercharacter}, Bergeron and Thiem describe a supercharacter theory in which both $\sch(U_n)$ and $\scl(U_n)$ are in bijection with sets partitions of $[n]=\{1,\ldots,n\}$. Through the section, given set partitions $\pi,\sigma\vdash[n]$ we will write $\chi^{\pi}$ for the supercharacter corresponding to $\pi$ and $\K_{\sigma}$ for the superclass corresponding to $\sigma$.
\smallskip

This supercharacter theory has an explicit formula for the supercharacter values:
\begin{proposition}
Let $\pi,\sigma\vdash[n]$, then 
\[\chi^{\pi}(\K_{\sigma})=\left\{\begin{array}{lr}
q^{\dim(\pi)-d(\pi)-\nst_{\pi}(\sigma)}\cdot(q-1)^{d(\pi)}\cdot (-\frac{1}{q-1})^{|D(\pi)\cap D(\sigma)|}&\mbox{if }D(\sigma)\subseteq\reg\pi;\\
0&\mbox{otherwise.}
\end{array}\right.\]
In particular, $\chi^{\pi}(1)=(q-1)^{d(\pi)}\cdot q^{\dim(\pi)-d(\pi)}$.
\end{proposition}
In \cite{baker2014antipode}, the authors describe $c(\pi)$ as
\[c(\pi)=\frac{q^{\crs(\pi)}(q-1)^{d(\pi)}}{\chi^{\pi}(1)},\]
so that
\[\langle \chi^{\pi},\chi^{\pi}\rangle=(q-1)^{d(\pi)}q^{\crs(\pi)}.\]
\begin{corollary}
Set $\pi\vdash[n]$, then 
\[\SPl_n(\chi^{\pi}):=\SPl_{U_n}(\chi^{\pi})=\frac{1}{q^{\frac{n(n-1)}{2}}}\frac{(q-1)^{d(\pi)}\cdot q^{2\dim(\pi)-2 d(\pi)}}{q^{\crs(\pi)}}.\]
\end{corollary}
\begin{proof}
This is a direct consequence of Definition \ref{definition of superplancherel}, since $|U_n|=q^{n(n-1)}/2$.
\end{proof}
Notice that this supercharacter theory is consistent, so the theory of Section \ref{section superinduction} applies.
\subsection{A combinatorial interpretation of the superplancherel measure}\label{section combinatorial interpretation}
We associate to $\pi\vdash[n]$ the following set $\mathcal{J}_{\pi}\subseteq U_n$: a matrix $A$ belongs to $\mathcal{J}_{\pi}$ if and only if
\begin{itemize}
 \item if $(i,j)\in D(\pi)$ then $A_{i,j}\in \mathbb{K}\setminus\{0\}$;
 \item if $(i,j)\in \reg(\pi)\setminus D(\pi)$ then $A_{i,j}=0$;
 \item if $(i,j)\in \Sing(\pi)$ then $A_{i,j}\in \mathbb{K}$.
\end{itemize}

 \begin{figure}[H]\begin{center}
\[\mathcal{J}_{\pi}=\left[\begin{array}{ccccc}
       1&\bullet&\bullet&\ast&0\\
       0&1&\ast&\bullet&0\\
       0&0&1&\bullet&\ast\\
       0&0&0&1&\bullet\\
       0&0&0&0&1
      \end{array}\right]\]
\end{center}
 \caption{Example of $\mathcal{J}_{\pi}$, where  $\pi=\{\{1,4\},\{2,3,5\}\}$. Here $\ast$ means that in that position there is an element of $\mathbb{K}^{\times}$, and $\bullet$ is an element of $\mathbb{K}.$}\label{example of J_pi}
\end{figure}
We say that a matrix $A$ in $\mathcal{J}_{\pi}$ is \emph{canonical} if
\[A_{i,j}=\left\{\begin{array}{lcl}1&\mbox{if}&(i,j)\in D(\pi)\mbox{ or }i=j\\0&&\mbox{otherwise}\end{array}\right.\]
In this section we show that the sets $\mathcal{J}_{\pi}$ partition of the group $U_n$ and that given a matrix $A\in U_n$ the set partition $\pi$ such that $A\in\mathcal{J}_{\pi}$ can be computed efficiently. We stress out that $\mathcal{J}_{\pi}$ are not the superclasses for this supercharacter theory, and in general they are not even union of conjugacy classes.

The following algorithm takes as input a matrix $A\in U_n$ and gives as output a canonical matrix $\tilde{A}\in\mathcal{J}_{\pi}$ for some $\pi$. The algorithm will consists of $n$ steps, at the step $k$ we will consider the $k$-th diagonal $d_k$ of $A^{k-1}$ starting from the upper-right corner, where 
\[d_k(A)=\{A_{1,n-k+1},A_{2,n-k+2},\ldots, A_{k,n}\}.\]

\begin{description}\label{lemma combinatorial algorithm}
\item STEP $0$: set $A^0=A.$
 \item STEP $1$: if $A^0_{1,n}=0$ set $A^1=A^0$;\\
  if $A^0_{1,n}\neq 0$ set $A^1_{1,n}=1$ and all other entries in the same row on the left and on the same column below $A^1_{1,n}$, up to the diagonal, equal to $0$. Set $A^1_{i,j}=A^0_{i,j}$ for all other entries $(i,j)$.
  \item STEP $k$: for all $A^{k-1}_{i,n-k+i}\in d_k(A^{k-1})$ do the following: if $A^{k-1}_{i,n-k+i}=0$ set $A^k_{i,j}=A^{k-1}_{i,j}$ for each $j=1,\ldots, n$; if $A^{k-1}_{i,n-k+i}\neq 0$ set $A^k_{i,n-k+i}=1$ and all other entries in the same row on the left and on the same column below $A^k_{i,n-k+i}$, up to the diagonal, equal to $0$. All other entries of $A^k$ are equal to the ones of $A^{k-1}$. 
\end{description}
For an example of the algorithm see Figure \ref{figure matrices}.

\begin{lemma}
 Given a matrix $A\in U_n$, there exists a unique $\pi$ such that $A\in \mathcal{J}_{\pi}$. In other words, $U_n=\bigsqcup\limits_{\pi\vdash[n]}\mathcal{J}_{\pi}$.
\end{lemma}
\begin{proof}
We start the proof with the following two observations:
\begin{itemize}
 \item consider $A\in\mathcal{J}_{\pi}$ and $(i,j)\in D(\pi)$, so that $A_{i,j}\neq0$. The matrix $\tilde{A}$ which is equal to $A$ except in the entry $\tilde{A}_{i,j}$, in which we still have $\tilde{A}_{i,j}\neq 0$, still belongs to $\mathcal{J}_{\pi}$;
 \item consider $A\in\mathcal{J}_{\pi}$ and $(i,j)\in D(\pi)$. Define $\tilde{A}$ such that all entries are the same as the entries of $A$, but those in the $i$-th row on the left of $(i,j)$, up to the diagonal, and those on the $j$-th column below $(i,j)$, up to the diagonal. These are the entries $\tilde{A}_{k,l}$ which correspond to the pairs $(k,l)\in \Sing(\pi)$. Hence $\tilde{A}\in\mathcal{J}_{\pi}$.
\end{itemize}
From these observations it is clear that $A^{k-1}\in\mathcal{J}_{\pi}$ if and only if $A^k\in\mathcal{J}_{\pi}$. Since the output of the algorithm is a canonical representative of $\mathcal{J}_{\pi}$, then it follows that for each matrix $A\in U_n$ there exists a unique $\pi\vdash[n]$ such that $A\in\mathcal{J}_{\pi}$.
\end{proof}
\begin{figure}
 \[A^0=\left[\begin{array}{ccccc}1&0&5&2&1\\&1&2&0&0\\&&1&5&0\\&&&1&4\\&&&&1\end{array}\right], A^1=A^0, A^2=\left[\begin{array}{ccccc}1&0&0&0&1\\&1&2&0&0\\&&1&5&0\\&&&1&0\\&&&&1\end{array}\right], A^3=A^2, A^4=\left[\begin{array}{ccccc}1&0&0&0&1\\&1&1&0&0\\&&1&1&0\\&&&1&0\\&&&&1\end{array}\right]\]\caption{Example of the algorithm described above; we start from a matrix $A=A^0\in U_n$ and we obtain a matrix $A^4$ which is canonical for the set $\mathcal{J}_{\pi}$ for $\pi=\{\{1,5\},\{2,3,4\}\}$. In general, if during the algorithm we find a non-zero term on the $k$-th diagonal then this corresponds necessarily to an arc of $\pi$ and not to a singular pair.}\label{figure matrices}
\end{figure}
\begin{proposition}
 For any $\pi\vdash[n]$, \[\SPl(\pi)=\frac{|\mathcal{J}_{\pi}|}{|U_n|}.\] Equivalently, the superplancherel measure of $\pi$ is the probability of choosing a random matrix in $U_n$ which belongs to $\mathcal{J}_{\pi}$.
\end{proposition}
\begin{proof}
 It is enough to prove that \[|\mathcal{J}_{\pi}|=\frac{(q-1)^{d(\pi)}\cdot q^{2\dim(\pi)-2 d(\pi)}}{q^{\crs(\pi)}};\] in order to do so we calculate $|\Sing(\pi)|$. Given a pair $(i,j)$ we write $\sigma^n_{(i,j)}\vdash[n]$ for the set partition such that $D(\sigma_{(i,j)}^n)=\{(i,j)\}$. Then \[\Sing(\sigma^n_{(i,j)})=\{(i,i+1),\ldots,(i,j-1),(i+1,j),\ldots,(j-1,j)\}.\] The cardinality $|\Sing(\sigma^n_{(i,j)})|$ is clearly $2(j-i-1)$. It is immediate to see that 
\[\Sing(\pi)=\bigcup_{(i,j)\in D(\pi)}\Sing(\sigma^n_{(i,j)}).\]
We use the inclusion-exclusion formula to calculate $|\Sing(\pi)|$: notice that for a pair of different arcs $(i,j),(k,l)$ with $i<k$ then 
\[\Sing(\sigma^n_{(i,j)})\cap\Sing(\sigma^n_{(k,l)})=\left\{\begin{array}{lcl}\{(j,k)\}&\mbox{if}&i<k<j<l  \\ \emptyset&&\mbox{otherwise}\end{array}\right.\]
Moreover, for any triplet of different arcs $(i,j),(k,l),(r,s)\in D(\pi)$ we have \[\Sing(\sigma^n_{(i,j)})\cap\Sing(\sigma^n_{(k,l)})\cap\Sing(\sigma^n_{(r,s)})=\emptyset.\] Thus
\[|\Sing(\pi)|=\sum_{(i,j)\in D(\pi)}|\Sing(\sigma^n_{(i,j)})|-\sum_{\substack{(i,j)\neq(k,l)\\\mbox{in }D(\pi)}}|\Sing(\sigma^n_{(i,j)})\cap\Sing(\sigma^n_{(k,l)})|,\]
hence
\[|\Sing(\pi)|=\sum_{(i,j)\in D(\pi)}2(j-i-1)-\sharp\{(i,j),(l,k)\in D(\pi)\mbox{ s.t. }i<l<j<k\}=2(\dim(\pi)-d(\pi))-\crs(\pi),\]
which concludes the proof.
\end{proof}
We use this interpretation to generate the second picture of Figure \ref{fig:boat1}: we generate a random matrix $A\in U_n$, then we apply the algorithm described above in order to reduce $A$ to a canonical matrix $\tilde{A}$. We define $\pi$ as the set partition whose arcs are exactly the non zero entries of $\tilde{A}$, so that $\tilde{A}\in\mathcal{J}_{\pi}$. Such a set partition is random distributed with the superplancherel measure.
\section{Set partitions as measures on the unit square}\label{section set partitions as measures}
In this section we will describe an embedding of set partitions into particular measures on a subset of the unit square $\Delta=\{(x,y)\in [0,1]^2\mbox{ s.t. }y\geq x\}$. We settle first some notation: if $A\subseteq \R^2$ is measurable then we write $\lambda_A$ for the uniform measure on $A$ of total mass equal to $1$, that is, $\int_A\,d\lambda_A=1$; given $n\in\N, i<j\leq n$, set 
\[A_{i,j}=\left\{(x,y)\in\R^2\mbox{ s.t. }\frac{i-1}{n}\leq x\leq \frac{i}{n}, \frac{j-1}{n}\leq y\leq \frac{j}{n}\right\}\subseteq\Delta;\]
for $\pi\vdash[n]$ we will write $A_{\pi}:=\bigcup_{(i,j)\in D(\pi)}A_{i,j}$ and $\mu_{\pi}=\frac{1}{n}\sum_{(i,j)\in D(\pi)}\lambda_{A_{i,j}}$. An example is given in Figure \ref{example of the measure mu_pi}.

\begin{figure}[H]
 \begin{center}
 \[\mu_{\pi}=\begin{array}{c}
\begin{tikzpicture}[scale=0.5]
\draw [thick, <->] (0,10) -- (0,0) -- (10,0);
\draw [dashed](0,0)--(9,9)--(0,9);
\draw [fill=lightgray](0,4) rectangle (1,5);
\draw [fill=lightgray](4,6) rectangle (5,7);
\draw [fill=lightgray](2,3) rectangle (3,4);
\draw [fill=lightgray](3,8) rectangle (4,9);
\draw [fill=lightgray](5,7) rectangle (6,8);
\draw [thick] (-.1,9) node[left]{1} -- (.1,9);
\draw [thick] (-.1,8) node[left]{8/9} -- (.1,8);
\draw [thick] (-.1,7) node[left]{7/9} -- (.1,7);
\draw [thick] (-.1,6) node[left]{6/9} -- (.1,6);
\draw [thick] (-.1,5) node[left]{5/9} -- (.1,5);
\draw [thick] (-.1,4) node[left]{4/9} -- (.1,4);
\draw [thick] (-.1,3) node[left]{3/9} -- (.1,3);
\draw [thick] (-.1,2) node[left]{2/9} -- (.1,2);
\draw [thick] (-.1,1) node[left]{1/9} -- (.1,1);
\draw [thick] (-.1,0) node[left]{0} -- (.1,0);
\draw [thick] (9,-.1) node[below]{1} -- (9,.1);
\draw [thick] (8,-.1) node[below]{$\frac{8}{9}$} -- (8,.1);
\draw [thick] (7,-.1) node[below]{$\frac{7}{9}$} -- (7,.1);
\draw [thick] (6,-.1) node[below]{$\frac{6}{9}$} -- (6,.1);
\draw [thick] (5,-.1) node[below]{$\frac{5}{9}$} -- (5,.1);
\draw [thick] (4,-.1) node[below]{$\frac{4}{9}$} -- (4,.1);
\draw [thick] (3,-.1) node[below]{$\frac{3}{9}$} -- (3,.1);
\draw [thick] (2,-.1) node[below]{$\frac{2}{9}$} -- (2,.1);
\draw [thick] (1,-.1) node[below]{$\frac{1}{9}$} -- (1,.1);
 \end{tikzpicture}\end{array}\]
 \end{center}
 \caption{Example of the measure $\mu_{\pi}$ on $\Delta$ corresponding to $\pi=\left\{\{1,5,7\},\{2\},\{3,4,9\},\{6,8\}\right\}\vdash[9]$ of Figure \ref{picture standard representation}. Everywhere but the gray areas has zero weight, while the gray areas represent where the measure has uniform weight. Each square has total weight $\frac{1}{n}=\frac{1}{9}$, so that the total weight is $\int_{\Delta}\,d\mu=\frac{5}{9}$.}\label{example of the measure mu_pi}
 \end{figure}
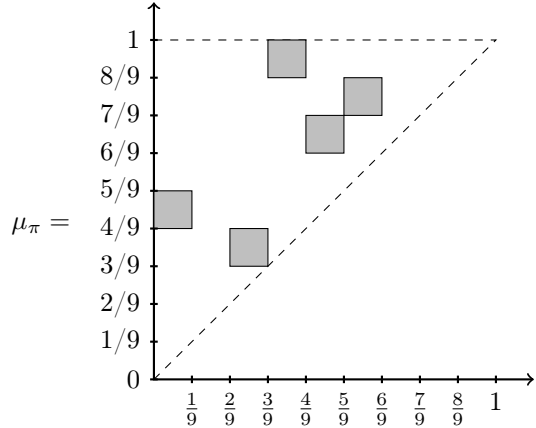
\begin{defi}
Let $X\subseteq \R^2$, set $\pi_1$ (resp. $\pi_2$) the projection into the first (resp. the second) coordinate. A measure $\mu$ on $X$ is said to have \emph{uniform marginals} if for each interval $I\subseteq \pi_1(X)$ and $J\subseteq\pi_2(X)$
\[\mu(I\times \pi_2(X))=|I|,\]
\[\mu(\pi_1(X)\times J)=|J|.\]
Similarly, the measure $\mu$ has \emph{subuniform marginals} if for each interval $I\subseteq \pi_1(X)$ and $J\subseteq\pi_2(X)$
\[\mu(I\times \pi_2(X))\leq|I|,\]
\[\mu(\pi_1(X)\times J)\leq|J|.\]
\end{defi}
As a measure on $\Delta$, $\mu_{\pi}$ has subuniform marginal and in particular $\int_{\Delta}\,d\mu\leq 1$. We call \emph{subprobability} a positive measure with total weight less than or equal to $1$, so that $\mu_{\pi}$ is a subprobability. We will sometimes deal with measures $\mu$ of $\Delta$ as measures on the whole square unit interval $[0,1]^2$, assuming that $\mu([0,1]^2\setminus \Delta)=0$.

\subsection{Statistics of set partitions approximated by integrals}
We define the following space of measures:
\[\Gamma:=\{\mbox{subprobabilities }\mu \mbox{ on }\Delta\mbox{ s.t. }\mu\mbox{ has subuniform marginals}\};\]
In this new setting we can describe the values of $d(\pi), \dim(\pi),\crs(\pi)$ as follows:
\begin{lemma}
Let $\pi\vdash[n]$, so that $\mu_{\pi}\in\Gamma$, then 
\begin{enumerate}
\item $d(\pi)\in O(n)$;
\item $\dim(\pi)=n^2\int_{\Delta}(y-x)\,d\mu_{\pi}(x,y);$
\item $\crs(\pi)=n^2\int_{\Delta^2}\mathbb{1}[x_1<x_2<y_1<y_2]\,d\mu_{\pi}(x_1,y_1)\,d\mu_{\pi}(x_2,y_2) + O(n).$
\end{enumerate}
\end{lemma}
\begin{proof}
 \begin{enumerate}
  \item It is immediate to see that $d(\pi)\leq n-1$, with equality if and only if $\pi=\{\{1,2,\ldots,n\}\}$.
  \item Since $\mu_{\pi}=\sum_{(i,j)\in D(\pi)}\mu_{\{\{i,j\}\}}$, it is enough to prove the statement for $\pi=\sigma^n_{(i,j)}$ (recall that $D(\sigma_{(i,j)}^n)=\{(i,j)\}$). Notice moreover that for $f\colon \R^2\to \R$ measurable
  \[\int_{A_{i,j}} f(x,y)\,d\lambda_{A_{i,j}}=\frac{\int_{A_{i,j}} f(x,y)\,dx\,dy}{\int_{A_{i,j}} \,dx\,dy}=n^2\int_{A_{i,j}} f(x,y)\,dx\,dy.\]
Therefore
  \begin{align*}
   n^2\int_{\Delta}(y-x)\,d\mu_{\pi}(x,y)&=n\int_{\Delta}(y-x)\,d\lambda_{A_{i,j}}\\&= n^3\int_{A_{i,j}} (y-x)\,dx\,dy\\&=j-i\\&=\dim(\pi).
  \end{align*}
\item Similarly as before, we have that for $A,B$ bounded subsets of $\R^2$ and $f\colon \R^4\to \R$
\[\int_{A\times B} f(x_1,y_1,x_2,y_2)\,d\lambda_A(x_1,y_1)\,d\lambda_B(x_2,y_2)=\frac{\int_{A\times B} f(x_1,y_1,x_2,y_2)\,dx_1\,dy_1\,dx_2\,dy_2}{\int_A \,dx\,dy\int_B\,dx\,dy}.\]
We see that
\begin{align*}
 n^2\int_{\Delta^2}\mathbb{1}[x_1<x_2<y_1<y_2]&\,d\mu_{\pi}(x_1,y_1)\,d\mu_{\pi}(x_2,y_2)\\
 &=\sum_{(i,j),(k,l)\in D(\pi)}\int_{\Delta^2}\mathbb{1}[x_1<x_2<y_1<y_2]\,d\lambda_{A_{(i,j)}}(x_1,y_1)\,d\lambda_{A_{(k,l)}}(x_2,y_2)\\
 &=n^4\sum_{(i,j),(k,l)\in D(\pi)}\int_{A_{(i,j)}\times A_{(k,l)}}\mathbb{1}[x_1<x_2<y_1<y_2]\,dx_1\,dy_1\,dx_2\,dy_2.\\
\end{align*}
Suppose $(i,j),(k,l)\in D(\pi)$ and call $\mathcal{I}[i,j,k,l]=\int_{A_{(i,j)}\times A_{(k,l)}}\mathbb{1}[x_1<x_2<y_1<y_2]\,dx_1\,dy_1\,dx_2\,dy_2,$ direct computations show that
\begin{itemize}
 \item if $i<k<j<l$ then $\mathcal{I}[i,j,k,l]=\frac{1}{n^4}$;
 \item if $i<j=k<l$ then $\mathcal{I}[i,j,k,l]=\frac{1}{2n^4}$;
 \item if $i=k<j=l$ then $\mathcal{I}[i,j,k,l]=\frac{1}{4n^4}$;
 \item in any other case $\mathcal{I}[i,j,k,l]=0$.
\end{itemize}
Hence
\[\crs(\pi)=n^2\int_{\Delta^2}\mathbb{1}[x_1<x_2<y_1<y_2]\,d\mu_{\pi}(x_1,y_1)\,d\mu_{\pi}(x_2,y_2)-\frac{1}{4}d(\pi)-\frac{1}{2}\#\{\mbox{adjacent arcs of }\pi\},\]
where a pair of arcs $(i,j),(k,l)\in D(\pi)$ is \emph{adjacent} if $j=k$. Since the number of adjacent arcs is obviously less than the number of arcs, the second and the third terms of the RHS are $O(n)$, and we conclude.\qedhere\end{enumerate}
 \end{proof}

We can therefore write 
\begin{align*}
 &\SPl_{n}(\chi^{\pi})=\frac{1}{q^{\frac{n(n-1)}{2}}}\frac{q^{2\dim(\pi)-2 d(\pi)}}{(q-1)^{d(\pi)}q^{\crs(\pi)}}=\\
 &\exp\left(\log q\left(-\frac{n^2}{2}+\frac{n}{2}+\frac{\log (q-1)}{\log q}d(\pi)-2d(\pi)+2\dim(\pi)-\crs(\pi)\right)\right)=\\
 &\exp\left(-n^2\log q\left(\frac{1}{2}-2\int_{\Delta}(y-x)\,d\mu_{\pi}(x,y) +\int_{\mathrlap{\Delta^2}}\mathbb{1}[x_1<x_2<y_1<y_2]\,d\mu_{\pi}(x_1,y_1)\,d\mu_{\pi}(x_2,y_2)\right)+ O(n)\right). 
\end{align*}
For each measure $\mu\in\Gamma$ we set thus 
\begin{itemize}
 \item $I_1(\mu):=\int_{\Delta}(y-x)\,d\mu$;
 \item $I_2(\mu):=\int_{\Delta^2}\mathbb{1}[x_1<x_2<y_1<y_2]\,d\mu(x_1,y_1)\,d\mu(x_2,y_2)$;
 \item $I(\mu):=\frac{1}{2}-2I_1(\mu)+I_2(\mu)$.
\end{itemize}
Hence for $\pi\vdash[n]$ we have
\begin{equation}\label{formula of superplancherel with entropy}
 \SPl_{n}(\chi^{\pi})=\exp\left(-n^2\log q\cdot I(\mu_{\pi})+O(n)\right).
\end{equation}
\subsection{Maximizing the entropy}
We set
\[\tilde{\Gamma}:=\{\mbox{subprobabilities }\mu \mbox{ on }[0,1/2]\times[1/2,1]\mbox{ s.t. }\mu\mbox{ has uniform marginals}\}.\]
Recall that for a measurable function $f$ and a measure $\mu$ the push forward is
\[f_{\ast}\mu(A):=\mu(f_{\ast}^{-1}(A))\]
for each $A$ measurable. Consider $f(x)=1-x$ and the Lebesgue measure $\Leb([0,1/2])$ on the interval $[0,1/2]$. Define $\Omega$ as $\Omega:=f_{\ast}\Leb([0,1/2])$ and notice that $\Omega\in\tilde{\Gamma}$. The goal of this section is to prove the following proposition:
\begin{proposition}\label{minimizing I}
 Consider $\mu\in\Gamma$, then $I(\mu)=0$ if and only if $\mu=\Omega$.
\end{proposition}
We will prove the proposition after studying the two functionals $I_1$ and $I_2$.
\begin{lemma}\label{maximizing I_1}
 Let $\mu\in\Gamma$, then $I_1(\mu)=\int_{\Delta}(y-x)\,d\mu\leq 1/4$, with equality if and only if $\mu\in\tilde{\Gamma}$.
\end{lemma}
\begin{proof}
 We show that for each $\mu\in\Gamma$ there exists a measure $\tilde{\mu}\in\Gamma$ and intervals $I_{\mu}\subseteq[0,1]$ and $J_{\mu}\subseteq[0,1]$ such that $\tilde{\mu}$ has uniform marginals as a measure of $I_{\mu}\times J_{\mu}$ and $I_1(\mu)\leq I_1(\tilde{\mu})$. This will be proved by ``squeezing'' the measure $\mu$ toward the top-left corner of $\Delta$. Set $I_{\mu}=[0,\mu(\Delta)]$, $J_{\mu}=[1-\mu(\Delta),1]$, $f_{\mu}(x)=\mu([0,x]\times [0,1])\leq x$ and $g_{\mu}(y):=1-\mu([0,1]\times[1-y,1])\geq y$. We define the measure $\tilde{\mu}$ as the push forward of $\mu$ by the function $(x,y)\to(f_{\mu}(x),g_{\mu}(y))$. It is evident that $\tilde{\mu}([0,1]^2\setminus\Delta)=0$. By construction, $\tilde{\mu}$ has uniform marginals on $I_{\mu}\times J_{\mu}$. Therefore
 \begin{align*}
  I_1(\tilde{\mu})&=\int_{\Delta}(v-u)\,d\tilde{\mu}(u,v)\\
  &=\int_{I_{\mu}\times J_{\mu}}(v-u)\,d\tilde{\mu}(u,v)\\
  &=\int_\Delta(g_{\mu}(y)-f_{\mu}(x))\,d\mu(x,y)\geq\int_{\Delta}(y-x)\,d\mu(x,y)=I_1(\mu),  
 \end{align*}
where the inequality comes from $g_{\mu}(y)\geq y$ and $f_{\mu}(x)\leq x$. Notice that we have $I_1(\tilde{\mu})=I(\mu)$ if and only if $f_{\mu}(x)=x$ and $g_{\mu}(y)=y$ almost everywhere according to the marginal of $\mu$ in, respectively, the $x$ and $y$ coordinates. This is equivalent to $\tilde{\mu}=\mu$. For an example of this construction, see Figure \ref{picture  for the proof of I_1}.
 \smallskip
 
 Set $l_{\mu}=\mu(\Delta)$. We show that $I_1(\tilde{\mu})=l_{\mu}(1-l_{\mu})$. We write $I_1(\tilde{\mu})=\int_{\Delta}y\,d\tilde{\mu}-\int_{\Delta}x\,d\tilde{\mu}$ and consider the two integrals separately. Observe that the $y$-marginal of $\tilde{\mu}$ is $\Leb([1-l_{\mu},1])$, the Lebesgue measure on the interval $[1-l_{\mu},1]$; hence
 \[\int_{y=0}^{y=1}y\int_{x=0}^{x=1}\,d\tilde{\mu}(x,y)=\int_{1-l_{\mu}}^1 y\,dy=l_{\mu}-\frac{l_{\mu}^2}{2}.\]
Similarly (the $x$-marginal of $\tilde{\mu}$ is $\Leb([0,l_{\mu}])$)
 \[\int_{x=0}^{x=1}x\int_{y=0}^{y=1}\,d\tilde{\mu}(x,y)=\int_0^{l_{\mu}}x\,dx=\frac{l_{\mu}^2}{2}.\]
 Therefore $I_1(\tilde{\mu})=\int_{\Delta}y\,d\tilde{\mu}-\int_{\Delta}x\,d\tilde{\mu}=l_{\mu}-l_{\mu}^2.$

Since $l_{\mu}\leq 1$, the maximal value of $l_{\mu}(1-l_{\mu})$ is obtained when $l_{\mu}=1/2$, in which case $I_1(\tilde{\mu})=1/4$. We showed thus that for each measure $\mu\in\Gamma$ there exists a measure $\tilde{\mu}\in\Gamma$ such that $I_1(\mu)\leq I_1(\tilde{\mu})\leq 1/4$, with equality if and only if $\mu\in\tilde{\Gamma}$, which concludes the proof.
\end{proof}

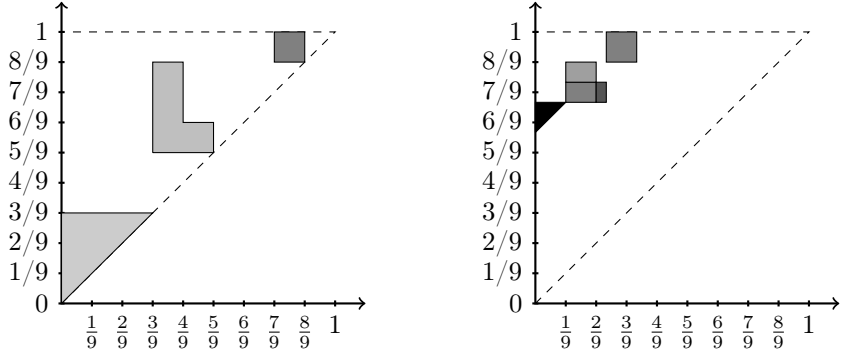
\begin{figure}
 \begin{center}
 \[ \begin{tikzpicture}[scale=0.4]
\draw [thick, <->] (0,10) -- (0,0) -- (10,0);
\draw [dashed](0,0)--(9,9)--(0,9);
\draw [thick] (-.1,9) node[left]{1} -- (.1,9);
\draw [thick] (-.1,8) node[left]{8/9} -- (.1,8);
\draw [thick] (-.1,7) node[left]{7/9} -- (.1,7);
\draw [thick] (-.1,6) node[left]{6/9} -- (.1,6);
\draw [thick] (-.1,5) node[left]{5/9} -- (.1,5);
\draw [thick] (-.1,4) node[left]{4/9} -- (.1,4);
\draw [thick] (-.1,3) node[left]{3/9} -- (.1,3);
\draw [thick] (-.1,2) node[left]{2/9} -- (.1,2);
\draw [thick] (-.1,1) node[left]{1/9} -- (.1,1);
\draw [thick] (-.1,0) node[left]{0} -- (.1,0);
\draw [thick] (9,-.1) node[below]{1} -- (9,.1);
\draw [thick] (8,-.1) node[below]{$\frac{8}{9}$} -- (8,.1);
\draw [thick] (7,-.1) node[below]{$\frac{7}{9}$} -- (7,.1);
\draw [thick] (6,-.1) node[below]{$\frac{6}{9}$} -- (6,.1);
\draw [thick] (5,-.1) node[below]{$\frac{5}{9}$} -- (5,.1);
\draw [thick] (4,-.1) node[below]{$\frac{4}{9}$} -- (4,.1);
\draw [thick] (3,-.1) node[below]{$\frac{3}{9}$} -- (3,.1);
\draw [thick] (2,-.1) node[below]{$\frac{2}{9}$} -- (2,.1);
\draw [thick] (1,-.1) node[below]{$\frac{1}{9}$} -- (1,.1);
\draw [fill={rgb:black,1;white,4}] (0,0) -- (3,3) -- (0,3) -- (0,0);
\draw [fill={rgb:black,1;white,3}] (3,5) -- (5,5) -- (5,6) -- (4,6) -- (4,8) -- (3,8) -- (3,5);
\draw [fill={rgb:black,1;white,1}] (7,8) rectangle (8,9);
\end{tikzpicture} \qquad\qquad \begin{tikzpicture}[scale=0.4]
\draw [thick, <->] (0,10) -- (0,0) -- (10,0);
\draw [dashed](0,0)--(9,9)--(0,9);
\draw [thick] (-.1,9) node[left]{1} -- (.1,9);
\draw [thick] (-.1,8) node[left]{8/9} -- (.1,8);
\draw [thick] (-.1,7) node[left]{7/9} -- (.1,7);
\draw [thick] (-.1,6) node[left]{6/9} -- (.1,6);
\draw [thick] (-.1,5) node[left]{5/9} -- (.1,5);
\draw [thick] (-.1,4) node[left]{4/9} -- (.1,4);
\draw [thick] (-.1,3) node[left]{3/9} -- (.1,3);
\draw [thick] (-.1,2) node[left]{2/9} -- (.1,2);
\draw [thick] (-.1,1) node[left]{1/9} -- (.1,1);
\draw [thick] (-.1,0) node[left]{0} -- (.1,0);
\draw [thick] (9,-.1) node[below]{1} -- (9,.1);
\draw [thick] (8,-.1) node[below]{$\frac{8}{9}$} -- (8,.1);
\draw [thick] (7,-.1) node[below]{$\frac{7}{9}$} -- (7,.1);
\draw [thick] (6,-.1) node[below]{$\frac{6}{9}$} -- (6,.1);
\draw [thick] (5,-.1) node[below]{$\frac{5}{9}$} -- (5,.1);
\draw [thick] (4,-.1) node[below]{$\frac{4}{9}$} -- (4,.1);
\draw [thick] (3,-.1) node[below]{$\frac{3}{9}$} -- (3,.1);
\draw [thick] (2,-.1) node[below]{$\frac{2}{9}$} -- (2,.1);
\draw [thick] (1,-.1) node[below]{$\frac{1}{9}$} -- (1,.1);
\draw [fill={rgb:black,1}] (0,17/3) -- (1,20/3) -- (0,20/3) -- (0,17/3);
\draw [fill={rgb:black,1;white,1}] (1,20/3) rectangle (2,22/3);
\draw [fill={rgb:black,3;white,5}] (1,22/3) rectangle (2,8);
\draw [fill={rgb:black,2;white,1}] (2,20/3) rectangle (7/3,22/3);
\draw [fill={rgb:black,1;white,1}] (7/3,8) rectangle (10/3,9);
\end{tikzpicture}\]
 \end{center}
 \caption{Example of the transformation of $\mu$ (left image) into $\tilde{\mu}$ (right image). The weights of the measure $\mu$ are the following: the triangle shape has a total weight of $1/9$, the \textbf{L} shape has a weight of $4/27$, and the top right box has weight $1/9$. Notice that the measure $\mu$ restricted to the top right box has already uniform marginals, hence the box will not be squeezed by the transformation into $\tilde{\mu}$, but will just shift.}\label{picture for the proof of I_1}
 \end{figure}
An immediate consequence of the previous lemma is that $I(\mu)\geq 0$, and we have $I(\mu)=0$ if and only if 
\begin{itemize}
 \item $\mu\in\tilde{\Gamma}$,
 \item $I_2(\mu)=0.$
\end{itemize} 
\smallskip 

\begin{lemma}\label{proposition mu=lambda_omega}
 Let $\mu\in\tilde{\Gamma}$ such that $I_2(\mu)=0.$ Then $\mu=\Omega$. 
\end{lemma}
\begin{proof}
 Consider a variation of the distribution function for a measure $\rho\in\tilde{\Gamma}$:
 \[F_{\rho}(a,b):=\rho([0,a]\times[1-b,1]),\]
 for $a,b\in[0,1/2]$. To prove the lemma it is enough to show that 
 \[F_{\mu}(a,b)=F_{\Omega}(a,b)=\min(a,b).\]
Suppose $a\leq b$ (the other case is done similarly), and consider the three sets $S=[0,a]\times[1/2,1-b]$, $T=[0,a]\times [1-b,1],$ $Q=[a,1/2]\times[1-b,1]$ as in Figure \ref{example for proof of I_2}. We claim that $\int_S\,d\mu=0$; suppose this is not the case, then $\int_S\,d\mu>0$. Notice that since $\mu$ has uniform marginals on the square $[0,1/2]\times [1/2,1]$ then 
\[a=\int_{[0,a]\times [1/2,1]}\,d\mu=\int_{S\cup T}\,d\mu=\int_S\,d\mu+\int_T\,d\mu.\]
By a similar argument we have $b=\int_T\,d\mu+\int_Q\,d\mu,$ therefore $\int_Q\,d\mu=b-a+\int_S\,d\mu>0$. We consider thus 
\begin{align*}
 I_2(\mu)&=\int_{\Delta^2}\mathbb{1}[x_1<x<2<y_1<y_2]\,d\mu(x_1,y_1)\,d\mu(x_1,y_2)\\
 &\geq \int_{S\times Q}\mathbb{1}[x_1<x<2<y_1<y_2]\,d\mu(x_1,y_1)\,d\mu(x_1,y_2).
\end{align*}
Observe that the characteristic function $\mathbb{1}[x_1<x_2<y_1<y_2]$ is equal to $1$ on the set $S\times Q$, hence $I_2(\mu)\geq\int_{S\times Q}\,d\mu\otimes d\mu=\mu(S)\cdot\mu(Q)>0$, which is a contradiction. Thus $\int_S\,d\mu=0$ as claimed. This implies that
\[F(a,b)=\mu(T)=\mu(T\cup S)=a\]
and the proof is concluded. 
\end{proof}
\begin{figure}

 \begin{center}
 \[\begin{tikzpicture}[scale=0.5]
\draw [thick, <->] (0,10) -- (0,0) -- (10,0);
\draw [dashed](0,0)--(9,9)--(0,9);
\draw (0,4.5) rectangle (2,6);
\draw [fill=lightgray](0.2,4.7) rectangle (0.45,5.1);
\node at (1,5){S};
\draw (0,6) rectangle (2,9);
\node at (1,7){T};
\draw (2,6) rectangle (4.5,9);
\draw [fill=lightgray](2.5,6.3) rectangle (2.9,6.7);
\node at (3.25,7){Q};
\draw [thick] (-.1,4.5) node[left]{1/2} -- (.1,4.5);
\draw [thick] (-.1,6) node[left]{1-b} -- (.1,6);
\draw [thick] (4.5,-.1) node[below]{$\frac{1}{2}$} -- (4.5,.1);
\draw [thick] (2,-.1) node[below]{a} -- (2,.1);
 \end{tikzpicture}\]
 \end{center}
 \caption{Example of the area division in the proof of Lemma \ref{proposition mu=lambda_omega}. Here we picture $a=2/9\leq b=3/9$. If the measure $\mu$ has non zero weight inside $S$ (here is pictured as the gray area), then it has also non zero weight in $Q$, and therefore $I_2(\mu)\neq 0$.}\label{example for proof of I_2}
 \end{figure}
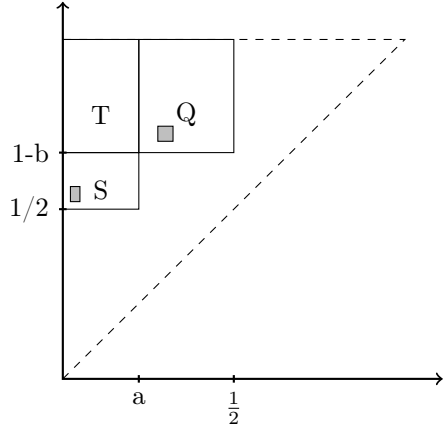
\begin{proof}[Proof of Proposition \ref{minimizing I}]
 It is easy to see that $I(\Omega)=0$. Suppose on the other hand that $I(\mu)=0$, then $I_1(\mu)=\frac{1}{4}+\frac{I_2(\mu)}{2}\leq \frac{1}{4}$ by Lemma \ref{maximizing I_1}. This implies that $I_2(\mu)=0$ and thus $I_1(\mu)=1/4$; hence $\mu\in\tilde{\Gamma}$ by Lemma \ref{maximizing I_1}, and we can apply Lemma \ref{proposition mu=lambda_omega} to conclude that $\mu=\Omega$.
\end{proof}

\section{Convergence in the weak* topology}\label{section final results}
In this section we prove the main result of the paper, that is, that in the weak* topology $\mu_{\pi^{(n)}}$ converges almost surely to $\Omega$ when $\pi^{(n)}$ is a random set partition distributed with the superplancherel measure. In order to do this we show some necessary lemmas, which relate the entropy $I$ to the L\'evy-Prokhorov metric on measures. We proceed as following: we show that the space $M^{\leq 1}(\Delta):=\{\mu\mbox{ measure on }\Delta\mbox{ s.t. }\int_{\Delta}\,d\mu\leq 1\}$ of subprobabilities on $\Delta$ is compact, and then we verify that $\Gamma$ is closed in $M^{\leq 1}(\Delta)$. We check then that both $I_1$ and $I_2$ are continuous as functions $\Gamma\to\R$. The proofs are mostly based on known theorems regarding probabilities, adapted in our case to subprobabilities.

Throughout the section, consider $(X,|\cdot|)$ a metric space and let $\mathcal{C}(X,\R)$ the set of continuous bounded functions $X\to\R$.
\begin{defi}
Let $(X,|\cdot|)$ be a metric space, then $M^{\leq 1}(X)$ and $M^1(X)$ are respectively the space of subprobabilities on $X$ and probabilities on $X$.
\end{defi}
We endow both $M^{\leq 1}(X)$ and $M^1(X)$ with the weak* topology, that is, consider $\{\mu_n\}_{n\in\mathbb{N}}\subseteq M^{\leq 1}(X),$ $\mu\in M^{\leq 1}(X)$ (resp. $\{\mu_n\}_{n\in\mathbb{N}}\subseteq M^1(X),\mu\in M^1(X)$), then we say that $\mu_n\overset{w^*}\to\mu$ in $M^{\leq 1}(X)$ (resp. $M^1(X)$) if
\[\int f(x) \,d\mu_n(x)\to\int f(x) \,d\mu(x)\]
for each $f\in\mathcal{C}(X,\R)$.
\smallskip

For a subset $Y\subseteq X$ and $\epsilon>0$ the $\epsilon-$neighborhood of $Y$ is
\[Y^{\epsilon}:=\{x\in X\mbox{ s.t. there exists }y\in Y\mbox{ with }\|x-y\|<\epsilon\}.\]
The L\'evy-Prokhorov metric is defined as
\[d_{L-P}(\mu,\nu)=\inf\{\epsilon>0\mbox{ s.t. }\mu(Y)\leq \nu(Y^{\epsilon})+\epsilon\mbox{ and }\nu(Y)\leq \mu(Y^{\epsilon})+\epsilon\mbox{ for each }Y\subseteq X\mbox{ measurable}\}.\]
It is well known that convergence with the L\'evy-Prokhorov metric is equivalent to the weak* convergence. For an introduction on the subject, see \cite{billingsley2013convergence}.
\smallskip 

As shown, for example, in \cite[Theorem 29.3]{billingsley2012probability}, if $X$ is compact then $M^1(X)$ is compact. The same is true for $M^{\leq 1}(X)$.
\begin{lemma}
Let $X$ be a compact metric space, then $M^{\leq 1}(X)$ is compact according to the weak* topology.
\end{lemma}
\begin{proof}
We can encode subrobabilities on $X$ as probabilities on $X\cup\{\partial\}$, where $\partial\notin X$ is called a \emph{cemetery} point, as follows:
\[\phi\colon M^{\leq 1}(X)\to M^{1}(X\cup\{\partial\})\]
\[\phi(\mu)(A)=\mu(A)\mbox{ if }A\subseteq X\mbox{ and }\phi(\mu)(\partial)=1-\mu(X).\]
Then $\phi$ is clearly an homeomorphism (with the obvious topology on $X\cup\{\partial\}$). Since $X\cup\{\partial\}$ is compact, then so is $M^1(X\cup\{\partial\})$, and thus also $M^{\leq 1}(X).$
\end{proof}

\begin{lemma}
 The set $\Gamma:=\{\mbox{subprobabilities }\mu \mbox{ on }\Delta\mbox{ s.t. }\mu\mbox{ has subuniform marginals}\}$ is closed in the set of subprobabilities $M^{\leq 1}(\Delta)$. In particular, $\Gamma$ is compact.
\end{lemma}
\begin{proof}
 Let $\{\mu_n\}$ be a sequence in $\Gamma$ converging to $\mu\in M^{\leq 1}(\Delta)$, we prove that $\mu\in\Gamma$. Suppose the contrary, we set without loss of generality that $\mu$ is not subuniform in the $x$-coordinates. Then there exists $(a,b)\in\Delta$ such that 
 \[\int_{[a,b]\times[0,1]}\,d\mu=b-a+\delta\]
 for $\delta>0$. Set $K=[a,b]\times[0,1]$ and $U=(a-\delta/3,b+\delta/3)\times[0,1]$. 
 By Uyshion's Lemma there exists a function $f\in\mathcal{C}(\Delta,[0,1])$ such that $\mathbb{1}_K(x,y)\leq f(x,y)\leq \mathbb{1}_U(x,y)$ for each $(x,y)\in\Delta$. Hence 
 \[\int_{\Delta}f(x,y)\,d\mu_n\leq\int_{[a-\frac{\delta}{3},b+\frac{\delta}{3}]\times[0,1]}\,d\mu_n\leq b-a+\delta-\frac{2}{3}.\]
 Since $\int f\,d\mu_n\to\int f\,d\mu$, this implies that $\int_{\Delta}f(x,y)\,d\mu\leq b-a+\delta-2/3$, contradiction. 
\end{proof}
The following lemma can be found in \cite[Theorem 29.1]{billingsley2012probability}:
\begin{lemma}
 Let $f\in\mathcal{C}(\Delta,\R)$ be bounded, then the functional that maps $\mu$ to $\int f\,d\mu$ is continuous. In particular, $I_1(\mu)=\int_{\Delta}(y-x)\,d\mu$ is continuous.
\end{lemma}
To prove the continuity of $I_2$ we need the following proposition, which can be found in \cite[Theorem 29.2]{billingsley2012probability}:
\begin{proposition}
 Suppose that $h\colon\R^k\to\R^j$ is measurable and that the set $D_h$ of its discontinuities is measurable. If $\nu_n\to\nu$ in $\R^k$ and $\nu(D_h)=0$, then $h_{\ast}\nu_n\to h_{\ast}\nu$ in $\R^j$.
\end{proposition}
\begin{lemma}
 The functional $I_2(\mu)=\int_{\Delta^2}\mathbb{1}[x_1<x<2<y_1<y_2]\,d\mu(x_1,y_1)\,d\mu(x_1,y_2)$ is continuous.
\end{lemma}
\begin{proof}
 We prove that $I_2$ is sequentially continuous, \emph{i.e.}, if $\mu_n\to\mu$ in weak* topology then $I_2(\mu_n)\to I_2(\mu)$. It is well known that in metric spaces continuity is equivalent to sequential continuity. Consider thus $\mu_n\to\mu$, then $\mu_n\otimes\mu_n\to\mu\otimes\mu$. Define the function $h\colon\Delta\times\Delta\to\R$, \[h(x_1,y_1,x_2,y_2):=\mathbb{1}[x_1<x_2<y_1<y_2].\] We claim that if $D_h$ is the set of discontinuities of $h$ then $\mu\otimes\mu(D_h)=0$. Indeed 
 \[D_h=\{(x_1,y_1,x_2,y_2)\in\Delta\times\Delta\mbox{ s.t. }x_1=x_2\mbox{ or }x_2=y_1\mbox{ or }y_1=y_2\};\]
Consider for example the set $\{(x_1,y_1,x_2,y_2)\in\Delta\times\Delta\mbox{ s.t. }x_1=x_2\}$. Since $\mu$ has subuniform marginals, $x_1$ and $x_2$ chosen with $\mu$ will be almost surely different and thus $\{(x_1,y_1,x_2,y_2)\in\Delta\times\Delta\mbox{ s.t. }x_1=x_2\}$ has measure $0$. The same holds for the cases $x_2=y_1$ and $y_1=y_2$.
\smallskip

By applying the previous proposition we have therefore that $h_{\ast}(\mu_n\otimes\mu_n)\to h_{\ast}(\mu\otimes\mu)$. In particular
\[I_2(\mu_n)=h_{\ast}(\mu_n\otimes\mu_n)(1)\to h_{\ast}(\mu\otimes\mu)(1)=I_2(\mu).\qedhere\]
 \end{proof}
 As a consequence, the functional $I(\mu):=\frac{1}{2}-2I_1(\mu)+I_2(\mu)$ is continuous.
 \begin{theorem}
  We have
  \[\SPl(\{\pi\vdash [n]\mbox{ s.t. }d_{L-P}(\mu,\Omega)>\epsilon\})\to 0.\]
 \end{theorem}
\begin{proof}
We claim that for each $\epsilon>0$ there exists $\delta>0$ such that if $d_{L-P}(\mu,\Omega)>\epsilon$ then $|I(\mu)|> \delta$. Fix $\epsilon>0$ and suppose the claim not true, so that for each $\delta>0$ there is $\mu_{\delta}$ with $d_{L-P}(\mu_\delta,\Omega)>\epsilon$ and $|I(\mu)|\leq \delta$. Set $\delta=1/n$, we obtain a sequence $(\mu_n)$ with $|I(\mu_n)|\leq 1/n$. Since $\Gamma$ is compact there exists a converging subsequence $(\mu_{i_n})$. Call $\overline{\mu}$ the limit of this subsequence. Since $I$ is continuous we have $I(\overline{\mu})=\lim_n I(\mu_{i_n})=0$. This is a contradiction, since $\Omega$ is the unique measure in $\Gamma$ with $I(\Omega)=0$, and the claim is proved.
\smallskip

 Fix $\epsilon>0$, then there exists $\delta>0$ such that if $d_{L-P}(\mu,\Omega)>\epsilon$ then $|I(\mu)|> \delta$. Define the set $N^n_{\epsilon}:=\{\pi\vdash [n]\mbox{ s.t. }d_{L-P}(\mu,\Omega)>\epsilon\}$, then 
 \[\SPl(N^n_{\epsilon})=\sum_{\pi\in N^n_{\epsilon}}\exp(-n^2\log q I(\mu_{\pi})+O(n)).\]
 Recall that the number of set partitions of $n$, called the Bell number, is bounded from above by $n^n$; therefore
 \[\SPl(N^n_{\epsilon})\leq n^n\sup_{\pi\in N^n_{\epsilon}}\exp(-n^2\log q I(\mu_{\pi})+O(n))<\exp(-n^2\delta\log q+O(n\log n))\to 0.\qedhere\] 
\end{proof}
We prove now Theorem \ref{main result 1} and Corollary \ref{main result 2}:
\begin{theorem}
  For each $n\geq 1$ let $\pi_n$ be a random set partition of $n$ distributed with the superplancherel measure $\SPl_n$, then
 \[\mu_{\pi_n}\to\Omega\mbox{ almost surely.} \] 
\end{theorem}
\begin{proof}
 As before, set $N^n_{\epsilon}:=\{\pi\vdash [n]\mbox{ s.t. }d_{L-P}(\mu,\Omega)>\epsilon\}$, so that \[\SPl(N^n_{\epsilon})<\exp(-n^2\delta\log q+O(n\log n)).\] Thus $\sum_n\SPl_n(N_n^{\epsilon})<\infty$ and we can apply the first Borel Cantelli lemma, which implies that $\limsup_nN_n^{\epsilon}$ has measure zero for each $\epsilon>0$, and therefore $\mu_{\pi_n}\to\Omega$ almost surely. 
\end{proof}

\begin{corollary}
  For each $n\geq 1$ let $\pi_n$ be a random set partition of $n$ distributed with the superplancherel measure $\SPl_n$, then
 \[\frac{\dim(\pi)}{n^2}\to\frac{1}{4}\mbox{ a.s.},\qquad \crs(\pi)\in O_P(n^2).\] 
\end{corollary}
\begin{proof}
 Define $N^n_{\epsilon,\dim}:=\{\pi\vdash [n]\mbox{ s.t. }|\frac{\dim(\pi)}{n^2}-\frac{1}{4}|>\epsilon\}$, then for each $\epsilon\in N^n_{\epsilon,\dim}$ we have $I(\mu_{\pi})>\epsilon$. Hence $\SPl(N^n_{\epsilon,\dim})<\exp(-n^2\epsilon\log q+O(n\log n))\to 0$. As before, this implies $\sum_n\SPl_n(N_n^{\epsilon,\dim})<\infty$ and thus $\frac{\dim(\pi)}{n^2}\to\frac{1}{4}$ almost surely. The crossing case is done similarly.
\end{proof}

\section{Acknowledgments}
The author is deeply grateful to Valentin F{\'e}ray for introducing him to the subject and the help received in the development of the paper.

This research was founded by SNSF grant SNF-149461: ``Dual combinatorics of Jack polynomials''.

\bibliography{courant}{}

\bibliographystyle{alpha}

\end{document}